\documentclass[a4paper,twoside,12pt]{article}

\usepackage{amsfonts}
\usepackage{mathrsfs}
\usepackage{amsmath}
\usepackage{amsmath,graphics}
\usepackage{amssymb}
\usepackage{indentfirst,latexsym,bm,amsmath,pstricks,amssymb,amsthm,graphicx,fancyhdr,float}
\usepackage[all]{xy}
\usepackage{amsthm}
\usepackage{amscd}
\usepackage[colorlinks,urlcolor=black]{hyperref}
\usepackage[center]{titlesec}
\titleformat{\section}{\centering\Large\bfseries}{\arabic{section}.}{1em}{}
\titleformat{\subsection}{\centering\large\bfseries}{\arabic{section}.\arabic{subsection}.}{1em}{}

\def\@evenhead{\thepage \hfill \shortauthor \hfill}
\def\@oddfoot{}
\def\@evenfoot{}
\newtheorem{lemma}{{\sc Lemma}}[section]
\newtheorem{corollary}[lemma]{{\sc Corollary}}
\newtheorem{proposition}[lemma]{{\sc Proposition}}
\newtheorem{thm}{{\sc Theorem}}
\newtheorem{theorem}{{\sc Theorem}}
\newtheorem{remark}[lemma]{{\sc Remark}}

\numberwithin{equation}{section}

\def\d{\partial}

\def\N{\mathcal N}
\def\g{\mathfrak{g}}

\def\wt{\widetilde}
\def\Aut {{\rm Aut\,}}
\def\Der {{\rm Der\,}}
\def\Ad {{\rm Ad\,}}

\begin{document}
\title{The nilpotent variety and invariant polynomial functions in the Hamiltonian algebra}
\author{Junyan Wei}
\date{}
%
%
%

\maketitle

\begin{abstract}
Premet has conjectured that the nilpotent variety of any finite-dimensional restricted Lie algebra is an irreducible variety.
In this paper, we prove this conjecture in the case of Hamiltonian Lie algebra.
and show that its nilpotent variety is normal and a complete intersection.
In addition, we generalize the Chevalley Restriction theorem to Hamiltonian Lie algebra. Accordingly, we give the generators of the invariant polynomial ring.
\end{abstract}

\vspace{0.2cm}
\renewcommand{\thethm}{\Alph{thm}}
\renewcommand{\thefigure}{\arabic{section}-\arabic{figure}}
\renewcommand{\theconjecture}{\hspace{-0.08cm}}
\section{Introduction}
In this paper we continue the study of nilpotent variety and invariant polynomial functions of restricted Lie algebras after the work \cite{WCL}.
The original motivation for the irreduciblity of nilpotent variety comes from Premet's conjecture.
Namely, for any finite-dimensional restricted Lie algebra $\g$ over an algebraically closed field $k$ of characteristic $p>0$, the variety $\N(\g)$ of nilpotent elements is irreducible~\cite[Conjecture\,1]{P1}.
Restricted Lie algebras are analogs of algebraic Lie algebras of characteristic zero.
It is well known that the variety of nilpotent elements in classical Lie algebras is irreducible (cf.\cite{J}).
If $\g$ is the Jacobson-Witt algebra $W_n$, Premet \cite{P2} obtained that $\N(W_n)$ is irreducible by identifying $\N(W_n)$ as the closure of one nilpotent orbit. Skryabin \cite{S} proved the irreducibility of $\N(\g)$ in the case where $\g=B_{2r}$ is a Poisson Lie algebra. The Poisson Lie algebra has a center coinciding with the scalars $k$, and the factor algebra $B_{2r}/k$ is isomorphic to a certain Hamiltonian Lie algebra $H_{2r}'$ (see Theorem \ref{main3}) whose commutator subalgebra $[H_{2r}', H_{2r}']=H_{2r}$ is simple.
In \cite{WCL} we solve the case for special Lie algebra $S_n$ by constructing an open dense subset of $\N(S_n)$.
Since a subvariety of an irreducible variety may not be again irreducible, we couldn't get the irreducibility of $\N(H_{2r})$ from $\N(B_{2r})$ directly.

In the present paper we consider the Lie algebra of Cartan type $H$.
Throughout this paper, we always assume that $n=2r$ is an even integer and the characteristic $p>3$.
Our first main result is the following theorem.

\begin{thm}\label{main1}
Premet's Conjecture is true for $H_n$, that is
the nilpotent variety of $H_n$ is irreducible. It is also normal and a complete intersection. The ideal $I_{\N(H_n)}=\{\varphi \in k[H_n] | \varphi(\N(H_n)) = 0\}$
is generated by $\xi_0,\,\ldots,\,\xi_{r-1},$
where the polynomials $\xi_i$ are defined in \eqref{xi_i}.
\end{thm}

The Chevalley restriction theorem
for the reductive Lie algebra was firstly generalized by Premet \cite{P2} for Cartan type, which is
\begin{equation}\label{iniso}
k[W_n]^{\Aut(W_n)} \cong k[T_0]^{{\rm GL}_n(\mathbb{F}_p)},
\end{equation}
where $T_0=\langle (1+x_1)\d/\d x_1,\cdots,(1+x_n)\d/\d x_n\rangle$ is a
torus of $W_n$.
In \cite{WCL}, we generalized the Chevalley restriction theorem to the special Lie algebra $S_n$.
We proved it by showing that there is an injective between $k[S_n]^{\Aut(S_n)}$ and $k[W_{n-1}]^{\Aut(W_{n-1})}$.
For the type H case, the authors in \cite{BFS} have already had that this restriction map is surjective.
In this paper, we will improve their result by proving that it is an isomorphism.
The method is the same as in \cite{WCL}.

Let
$$T_H=\langle (1+x_{r+1})\d_{r+1}-x_1\d_1,\cdots,(1+x_{r+2})\d_{r+2}-x_2\d_2,\cdots,(1+x_{2r})\d_{2r}-x_r\d_r \rangle$$
be a generic torus introduced \cite{BFS} in $H_n$. Our second main theorem is:
\begin{thm}\label{main2}
  The restriction map $r_1:\,k[H_n]^{\Aut(H_n)}\longrightarrow\,k[T_H]^{{\rm GL}_r(\mathbb{F}_p)}$ is an isomorphism. In particular, $k[H_n]^{\Aut(H_n)}$ is generated by $\xi_0,\xi_1,\cdots,\xi_{r-1}.$
\end{thm}

The paper is organized as follows. We first give a brief description of the
Hamiltonian algebra and the Poisson Lie algebra and their relations.
In section 2 and 3, we will prove Theorems \ref{main1} and \ref{main2} respectively.

\section{The Hamiltonian algebra}
In this section, we will give a brief description of the Hamiltonian Lie algebra.

Let $B_n=k[x_1,\cdots,x_n]/\langle x_1^p,\cdots,x_n^p \rangle$
be the truncated polynomial ring.  Let $\mathfrak{m}$ be the unique maximal ideal of $B_n$. The derivation algebra $W_n =\Der B_n$ is
a Lie algebra called the Jacobson-Witt algebra.  It is well-known that
this is a simple restricted Lie algebra and its
$p$-map is the standard $p$-th power of linear operators. Denote by $\d_i = \d/\d x_i \in W_n$ the partial derivative with respect to the variable $x_i$. Then $\{\d_1,\ldots,\d_n\}$ is a basis of the $B_n$-module $W_n$.

Consider the following differential form
$$\omega_H=\sum\limits_{i=1}^r dx_i\wedge dx_{i+r}.$$
Define $H_n^{''}=\{D\in W_n\,|\,D(\omega_H)=0\},$ which is clearly a subalgebra of $W_n$. For any $f\in B_n$, it can be easily checked that
\begin{equation}\label{dh}
D_H(f):=\sum\limits_{i=1}^r\big(\d_i(f)\d_{i+r}-\d_{i+r}(f)\d_i\big)
\end{equation}
indeed sends $\omega_H$ to 0. Define
$$[f,g]:=D_H(f)(g),~\forall~f,g\in B_n.$$
The multiplication $[\,,\,]$  makes $B_n$ into a Lie algebra, which is called a Poisson algebra. Accordingly, the map $D_H:B_n \longrightarrow W_n$  sending $f\in B_n$ to $D_H(f)$ is a Lie homomorphism.

Let $H_n'$ denote the image of $B_n$ under $D_H$. It is immediate thereof that the Lie center $\{f \in B_n \,|\, [f,B_n]=0\}$ coincides
with $k$, and $H_n'$ is isomorphic to $B_n/k$.
Put $H_n=H_n'^{(1)}$ be the derived algebra of $H_n'$.
 The following Theorem gives the property of $H_n$.

\begin{theorem}[\cite{SF}]\label{main3}
(1) $H_n=\langle D_H(x^{\textbf{a}})\,|\,0\leq  \textbf{a} < \tau \rangle$ and
$H_n'=\langle D_H(x^{\textbf{a}}) \rangle.$ Here $x^{\textbf{a}}=x_1^{a_1}\cdots x_n^{a_n}$ is the usual notation in $B_n$ and $\tau=(p-1,\cdots,p-1).$

(2) $H_n$ is simple and restricted, and its dimension is $p^n-2$.
\end{theorem}

Note that the map $D_H$ is linear, so every element in $H_n$ has the form $D_H(f)$ for some $f\in B_n.$

Moreover, there is a $p$-map on $B_n$ which is compatible with the $p$-map on $H_n'$ and satisfying $1^{[p]}=0$ and $f^{[p]}\in \mathfrak{m}^2$ for all $f\in \mathfrak{m}$. This kind of $p$-map is called normalized. As every two normalized $p$-maps on $B_n$ are conjugate. Thus, we fix a normalized $[p]$-structure on $B_n$. Given $f\in B_n$, we define $f^{[1]}=f$ and inductively, $f^{[p^i]}=(f^{[p^{i-1}]})^{[p]}$ for $i > 0$.

Denote by
\begin{equation}\label{defofkappa}
\kappa :B_n\longrightarrow k
\end{equation} the homomorphism of associative algebras with kernel $\mathfrak{m}$.
For each integer $a \geq 0$,
Denote by $a_0, a_1, a_2, \ldots $ the coefficients in the $p$-adic
expansion
\begin{equation}\label{defofa_i}
a=\sum_{i\geq 0}a_ip^i,\qquad \text{with~}0 \leq a_i <p.
\end{equation}
Let
\begin{eqnarray*}
f^{<a>}&=&\prod\limits_{i\geq 0}\big(f^{[p^i]}-\kappa(f^{[p^i]})\big)^{a_i}\in B_n\\
f^{[a]}&=&\prod\limits_{i\geq 0}(f^{[p^i]})^{a_i}\in B_n
\end{eqnarray*}

From \cite[Propositions\,3.2,\,3.5]{S}, there exist uniquely determined homogeneous polynomial functions $\widetilde {\varphi}_1,\ldots,\widetilde{\varphi}_{p^r-1}$ (resp.~ $\varphi_1,\ldots,\varphi_{p^r-1}$) in $k[B_n]$ such that
\begin{equation}\label{invariant}
f^{<p^r>}+\sum\limits_{a=1}^{p^r-1}\widetilde{\varphi}_a(f)f^{<p^r-a>}=0\quad\big(resp.\quad
f^{[p^r]}+\sum\limits_{a=1}^{p^r-1}\varphi_a(f)f^{[p^r-a]}=0\big).
\end{equation}
And there is also a relation between the polynomial functions $\widetilde{\varphi_i}$ and $\varphi_i$:
\begin{equation}\label{relation}
\widetilde{\varphi}_a(f)=\sum_{b=1}^a(-1)^{a-b}{{a-1}\choose{b-1}}\kappa(f^{[a-b]})
\cdot\varphi_b(f)~~\forall f\in B_n~\text{and}~ 0<a <p^n.
\end{equation}

\section{The nilpotent variety of $H_n$}
In this section, we will show that the nilpotent variety of $H_n$ is irreducible and normal and a complete intersection.
We first will give the invariant polynomial functions whose zero set is  $\N(H_n)$ and introduce an open subset in $H_n$.

Consider $B_n$ as the natural $W_n$-module and let's denote by $\chi_D(t)$ the characteristic polynomial of $D\in W_n$ as a linear transformation of $B_n$. It was proved by Premet \cite{P2} that,
 $$\chi_D(t)=t^{p^n}+\sum\limits_{i=0}^{n-1}\psi_i(D)t^{p^i},$$
where each $\psi_i$ is a homogeneous polynomial function on $W_n$.

\begin{lemma}[{\cite[Lemma\,6.3]{S}}]\label{nil}
If $f\in B_n$, then $\psi_i(D_H(f))=0$ and $\psi_{r+i}(D_H(f))=\widetilde \varphi_{p^r-p^i}(f)^{p^r}$ for all $0\leq i< r,$ where $\wt\varphi_{p^r-p^i}$'s
are those polynomial functions in \eqref{invariant}.
\end{lemma}

Denote by $\phi_i=\widetilde{\varphi}_{p^r-p^i}$ for short.
Recall that the map $D_H:\,B_n \longrightarrow W_n$ is defined by \eqref{dh}.
It induces a surjection from $B_n^{(1)}$ to $H_n$.
Hence there exists a linear map $\delta:\,H_n\longrightarrow\,B_n^{(1)}$
such that $D_H \circ \delta=id_{H_n}.$
As a matter of fact, $\delta$  is given by
$$\delta\big(D_H(f)\big)=f-\kappa(f).$$
Because $\delta$ is linear, its differential satisfies
\begin{equation}\label{differential-of-delta}
(d \delta)_{D_H(f)}\big(D_H(g)\big)=\delta\big(D_H(g)\big)~\text{for~}f,g\in B_n.
\end{equation}

Let
\begin{equation}\label{xi_i}
  \xi_i=\phi_i\circ \delta,\quad \forall ~0\leq i \leq r-1.
\end{equation}
In terms of Lemma \ref{nil}, the polynomial functions $\xi_i$'s
satisfy the identities
$$\xi_i(D)^{p^r}=\Big(\phi_i\big(\delta(D)\big)\Big)^{p^r}=
\psi_{r+i}\Big(D_H\big(\delta(D)\big)\Big)=\psi_{r+i}(D),~\forall~ D\in H_n,~0\leq i <r.$$
Thus $\xi_i^{p^r}=\psi_{r+i}$. Therefore, $\xi_i$ is $\Aut(H_n)$-invariant. The next Lemma is easy to check.
\begin{lemma}
The nilpotent variety  $\N(H_n)$ of $H_n$ is the zero set of polynomial functions $\xi_0,\ldots,\xi_{r-1}$.
\end{lemma}

Denote by
$$U=\{f\in B_n\,|\,f,f^{[p]},\ldots,f^{[p^{r-1}]} \text{are linearly independent modulo~} k+\mathfrak{m}^2 \}.$$
It is due to \cite[Proposition\,2.2]{S} that the subset $U$ is nonempty and Zariski open in $B_n$. Consider the following morphisms
\begin{center}\mbox{}
\xymatrix{
       H_n \ar@{->}[r]^-{\delta} & B_n^{(1)} \ar@{->}[r]^-{\epsilon} & B_n}
\end{center}
where $\epsilon$ is the natural embedding.
Let
\begin{equation}\label{V}
 V=\delta^{-1}\circ \epsilon^{-1}(U).
\end{equation}

Since both $\delta$ and $\epsilon$ are linear, $V$ is open in $H_n$. Let
$u=\sum\limits_{i=1}^{r}(-1)^{i-1}x_i \prod\limits_{i=1}^{i-1} x_{r+j}^{p-1}$.
From the proof of \cite[Proposition\,5.2]{S}, it follows that $u\in U$.
Thus $D_H(u)\in V$, which implies that $V$ is  also non-empty.

We begin to prove Theorem \ref{main1}.
The proof is based on the following two Lemmas. The first one is proved in
\cite[Lemma\,1.5]{S}. The second one will be proved at the end of this section.
\begin{lemma}\label{main-lemma}
Let $\N \subset L$ be the zero set of homogeneous polynomial functions
$\tau_1, \ldots,\tau_m\in k[L]$. Suppose that $M \subset L$ is an open subset and $E \subset L$ a vector subspace such that $E \cap \N \subset \{0\} \cup M$ and $(d\tau)_x,\ldots,(d\tau_m)_x$ are linearly independent
at all points $x \in M$.

{\rm(1)}. If $\dim E \geq m+1$, then $\N$ is a complete intersection of codimension $m$ in $L$ and the ideal $I_{\N}=\{\tau \in k[L] \,|\, \tau(\N) = 0\}$ is generated by $\tau_1,\ldots,\tau_m$.

{\rm(2)}. If $\dim E \geq m+2$, then $\N$ is normal and irreducible.
\end{lemma}

\begin{lemma}\label{linear-independent}
The differentials $d\xi_0,\ldots,d\xi_{r-1}$ are linearly
independent at all $x\in V$, where $\xi_i$'s are defined in \eqref{xi_i}
and $V$ is defined in \eqref{V}.
\end{lemma}

\begin{proof}[Proof of Theorem {\rm \ref{main1}}]
According to Lemmas \ref{main-lemma} and \ref{linear-independent},
it suffices to find a vector subspace $E\subseteq H_n$ with
$$\dim E \geq r+2,\quad\text{and}\quad E \cap \N(H_n) \subseteq \{0\}\cup V.$$

Let
$$
u=\sum\limits_{i=1}^{r}(-1)^{i-1}x_{i} \prod\limits_{i=1}^{i-1} x_{r+j}^{p-1};\quad
v=\sum\limits_{i=1}^{r}(-1)^{i-1}x_{r+i} \prod\limits_{i=1}^{i-1} x_{j}^{p-1};
$$
Define
$$E=\langle D_H(u),D_H(v),D_H(x_ix_{r+i})\,|\,i=1,\ldots,n\rangle \subseteq H_n.$$
It is easy to see that $\dim E=r+2$. It remains to prove that
$E\cap \N(H_n) \subseteq \{0\}\cup V.$

Choose any
$$D=\lambda D_H(u)+\mu D_H(v)+D_H(s)=D_H(\lambda u+\mu v +s)\in E,$$
where $s$ is a linear combination of $s_1,\ldots,s_n$.
In terms of the proof of Theorem 6.4 in \cite{S}, we have that
\begin{itemize}
\item[(1)]
If~ $\lambda\neq 0$ or $\mu \neq 0$,~ then ~$\lambda u+s +\mu u\in U$;
\item[(2)]
If ~$\lambda=\mu=0$,~ then~ $D=D_H(s)$~ is~ $[p]$-semisimple.
\end{itemize}
Thus if $D\in E\cap \N(H_n)$, then $D \in V$ or $D=0$, i.e.,
$E\cap \N(H_n) \subseteq \{0\}\cup V.$ Therefore we prove Theorem \ref{main1}.
\end{proof}

\begin{proof}[Proof of Lemma {\rm\ref{linear-independent}}]
Assume that there are some $\lambda_i\in k~(0\leq i \leq r-1)$ such that $\sum\limits_{i=0}^{r-1} \lambda_i (d\xi_i)_x=0$. We claim that
\begin{equation}\label{result}
\sum\limits_{i=0}^{r-1} \lambda_i (d\phi_i)_{\delta(x)}=0 ~\text{on~the~whole~algebra~}B_n.
\end{equation}
Assume that \eqref{result} is true. As $d\phi_i~(0\leq i \leq r-1)$
are linearly independent at $\delta(x)\in U$, it follows that
$\lambda_i=0$ for $0\leq i \leq r-1$. Hence our lemma is proved.
Therefore it is enough to prove \eqref{result}.

By the definition of differential,
\begin{eqnarray*}
(d\xi_i)_x(y)&=&\lim_{t \rightarrow 0}\frac{\xi_i(x+ty)-\xi_i(x)}{t}\\
&=& \lim_{t \rightarrow 0}\frac{\phi_i(\delta(x))+\phi_i(t\delta(y))-\phi_i(\delta(x))}{t}\\
&=&(d\phi_i)_{\delta(x)}(\delta(y))
\end{eqnarray*}
Therefore $(d\xi_i)_x=(d\phi_i)_{\delta(x)}\circ \delta$.

Note that $(f+\lambda)^{<a>}=f^{<a>}$ for any $\lambda\in k$, so $\widetilde{\varphi}_a(f+\lambda)=\widetilde{\varphi}_a(f)$.
Thus $$(d\phi_i)_f(\lambda)=\lim_{t \rightarrow   0}\frac{\phi_i(f+t\lambda)-\phi_i(f)}{t}=0.$$
So
\begin{equation}\label{partly}
  \sum\limits_{i=0}^{r-1} \lambda_i (d\phi_i)_{\delta(x)}\big |_{B_n^{(1)}}=0.
\end{equation}

Note that $\sum\limits_{i=0}^{r-1} \lambda_i (d\phi_i)_{\delta(x)}$ is linear on $B_n$.
To prove \eqref{result}, by \eqref{partly}, it suffices to show that
\begin{equation}\label{final}
\sum\limits_{i=0}^{r-1} \lambda_i (d\phi_i)_{\delta(x)}(x_1^{p-1}\cdots x_n^{p-1})=0.
 \end{equation}
Actually, we will show that for each $0\leq i\leq r-1$,
\begin{equation}\label{final2}
(d\phi_i)_{\delta(x)}(x_1^{p-1}\cdots x_n^{p-1})=0.
\end{equation}

It is easy to see that $(d\kappa)_f=\kappa$ as $\kappa$ is linear. Note that $$(f+tg)^{[a]}=f^{[a]}+t\sum\limits_{\{j\geq 0\,|\, p^j\leq a\}}a_j f^{[a-p^j]}D_H(f)^{p^j-1}(g)+\text{terms~divisible~by~} t^2.$$
Hence by (\ref{relation}), we have that
\begin{equation}\label{final3}
\hspace{-0.4cm}
\begin{aligned}
(d\widetilde{\varphi}_a)_{\delta(x)}(g)=&\sum\limits_{b=1}^a(-1)^{a-b}{{a-1}\choose{b-1}}
\cdot\bigg(\kappa\big({\delta(x)}^{[a-b]}\big)\cdot (d\varphi_b)_{\delta(x)}(g)+\\
&\varphi_b({\delta(x)})\cdot
 \kappa\Big(\sum\limits_{\{j\geq 0\,|\, p^j\leq a-b\}}(a-b)_j {\delta(x)}^{[a-b-p^j]}D_H({\delta(x)})^{p^j-1}(g)\Big)\bigg),
\end{aligned}
\end{equation}
where $(a-b)_j$ is the $j$-th coefficient in the $p$-adic
expansion (cf. \eqref{defofa_i}) of $(a-b)$.

According to \eqref{defofkappa} and \eqref{final3},
\eqref{final2} follows from the following two results (\,$\forall~1 \leq b \leq p^r-p^i~\text{and}~\forall~ 0\leq i\leq r-1$):
\begin{eqnarray}
\kappa\big(\delta(x)^{[p^r-p^i-b]}\big)\cdot (d\varphi_b)_{\delta(x)}
(x_1^{p-1}\cdots x_n^{p-1})=0,&&
\label{final4}\\[0.1cm]
D_H({\delta(x)})^{p^j-1}
(x_1^{p-1}\cdots x_n^{p-1})\in \mathfrak m.&&\forall~p^j \leq p^r-p^i-b.\label{final5}
\end{eqnarray}

Firstly, we prove \eqref{final4}.
Indeed, by \cite[Proposition\,3.2]{S}, we have
$$(d \varphi_b)_{\delta(x)}(x_1^{p-1}\cdots x_n^{p-1})=\varphi_1\big(\delta(x)^{[b-1]}x_1^{p-1}\cdots x_n^{p-1}\big).$$
If $p \nmid (b-1)$, then $\delta(x)^{[b-1]}\in \mathfrak{m}$ as $\delta(x)\in \mathfrak{m}$. And then
$$(d \varphi_b)_{\delta(x)}(x_1^{p-1}\cdots x_n^{p-1})=\varphi_1(0)=0.$$
If $p \,|\, (b-1)$, assume that $mp=b-1$, i.e., $b=mp-1$.
In this case we have that $p \nmid (p^r-p^i-b=p^r-p^i-mp-1)$ for any $0\leq i\leq r-1.$
Thus $\delta(x)^{[p^r-p^i-b]}\in \mathfrak{m}$,
i.e., $\kappa(\delta(x)^{[p^r-p^i-b]})=0.$
Therefore, \eqref{final4} is true.

It remains to prove \eqref{final5}.
Note that $D_H(\delta(x))^{p^j-1}=x^{p^j-1}=\prod\limits_{i=0}^{j-1}(x^{p^i})^{p-1}$. Thus
$$D_H\big(\delta(x)\big)^{p^j-1}(g)=\prod\limits_{i=0}^{j-1}D_H\big(\delta(x)^{[p^i]}\big)^{p-1}(x_1^{p-1}\cdots x_n^{p-1}).$$

Since $\deg D_H\big(\delta(x)\big)^{p^j-1}(x_1^{p-1}\cdots x_n^{p-1}) = \sum\limits_{i=1}^{j-1}(p-1)\deg \delta(x)^{[p^i]}+n(p-1)-2j(p-1)>0$,
thus $D_H(\delta(x))^{p^j-1}(g) \in \mathfrak{m}$.
The proof is complete.
\end{proof}

\section{The invariant polynomial ring of $H_n$}
Throughout this section, denote for short by $G_W=\Aut(W_r)$, $G_B$ the group of
 automorphisms of both the associative and
 the Lie algebra structures on $B_n$,
 and $G_H=\Aut(H_n)$ the automorphism group of $H_n$.
 In this section, we aim to prove Theorem \ref{main2}
 by showing an injective between $k[H_{2r}]^{G_H}$ and $k[W_r]^{G_W}$,
 then we get the result due to the commutative diagram (Figure \ref{commutative-diagram}).

It is well-known that
$$\{\mu \in \Aut B_n\,|\, \mu(\omega_H)\in k^*\omega_H\}\cong G_H,  \quad\text{and} \quad\{\mu \in \Aut B_n\,|\, \mu(\omega_H)=\omega_H\}=G_B  ,$$
where the isomorphism for $H_n$ is given by $\mu \mapsto \Ad_{\mu}$.
For any $D=\sum\limits_{i=1}^n f_i\d_i\in W_n$, the action of $\Ad_{\mu}$ on $W_n$ is
$$\Ad_{\mu}(D)=\sum\limits_{i,k=1}^n\Big(f_i\cdot \frac{\d \mu(x_k)}{\d x_i}\Big)\big(\mu^{-1}(x_1),\cdots,\mu^{-1}(x_n)\big)\d_k.$$

Now we consider the map
\begin{equation}\label{definiton-of-beta}
  \beta=D_H\circ \theta_r:\, W_r\longrightarrow H_n,
\end{equation}
where $\theta_r$ is defined by
$$\theta_r:\, W_r \longrightarrow \,B_n, \quad \sum\limits_{i=1}^r f_i\d_i \mapsto
\sum\limits_{i=1}^r x_if_i(x_{r+1},\cdots,x_{2r}).$$
It is proved  in \cite[Lemma\,4.3]{BFS} that $\beta$ is an injective homomorphism of restricted Lie algebras.
By pulling-back, $\beta$ induces a morphism
$$\beta^*:\, k[H_n] \longrightarrow k[W_r],\qquad
\beta^*(\rho)=\rho \circ \beta.$$

\begin{proposition}\label{propiso}
$\beta^*$ induces an injective morphism
\begin{equation}
\wt\beta=\beta^*|_{k[H_n]^{G_H}} :\, k[H_n]^{G_H} \longrightarrow \,k[W_r]^{G_W}.
\end{equation}

\end{proposition}
We leave the proof of Proposition \ref{propiso} at the end of this section.

\begin{proof}[Proof of Theorem {\rm\ref{main2}}]
Let
$$\begin{aligned}
T_W&=\langle (1+x_1)\d_1,\,\cdots,\,(1+x_r)\d_r \rangle \in W_r;\\
T_H&=\langle (1+x_{r+1})\d_{r+1}-x_1\d_1,
\,\cdots,\,(1+x_{2r})\d_{2r}-x_r\d_r \rangle \in H_n;
\end{aligned}
$$
be the generic tori defined in \cite{BFS}.
Then the map $\beta$ induces an isomorphism between $T_W$ and $T_H$.
Let us  consider the following diagram:
\begin{figure}
\begin{center}\mbox{ }
\xymatrix{
      k[H_n]^{G_H} \ar@{->}[d]^-{r_1}  \ar@{->}[r]^-{\wt\beta}
      &   k[W_r]^{G_W}  \ar@{->}[d]^-{r_2}\\
       ~~~~k[T_H]^{{\rm GL}_{r}(\mathbb{F}_{p})} \ar@{->}[r]^{\beta^*_{T_H}}
      &   k[T_W]^{{\rm GL}_{r}(\mathbb{F}_{p})}  }
\end{center}\vspace{-0.3cm}
\caption{commutative diagram \label{commutative-diagram}}
\end{figure}
Here, $r_1,r_2$ are both restriction maps; and $\beta^*_{T_H}$ is the restriction of $\beta^*$ on $k[T_H]^{{\rm GL}_{r}(\mathbb{F}_{p})}$. Clearly, the map $\beta^*_{T_H}$ is an isomorphism.

It is clear that it is a commutative diagram.
From \cite[Theorem\,6.12]{BFS} it follows that the map $r_1$ is surjective.
It remains to show that $r_1$ is injective.

By Proposition \ref{propiso}, $\wt\beta:\,k[H_n]^{G_H} \to k[W_r]^{G_W} $ is injective.
Hence the map $r_1$ is injective, because $r_2$ is an isomorphism \cite{P2}.

Note that, $k[T_H]^{{\rm GL}_{r}(\mathbb{F}_{p})}$ is generated by $\xi_0|_{T_H},\xi_1|_{T_H},\cdots,\xi_{r-1}|_{T_H}.$ Therefore, $k[H_n]^{G_H}$ is generated by $\xi_0,\xi_1,\cdots,\xi_{r-1}.$
\end{proof}
 \begin{remark}
Actually, the map $\wt\beta$ is an isomorphism.

  \end{remark}

Before going to prove Proposition \ref{propiso},
we first show that
\begin{equation}\label{contained}
  \wt\beta\left(k[H_n]^{G_H}\right) \subseteq  k[W_r]^{G_W}.
\end{equation}

To this end, it is enough to construct a map $G_W \rightarrow G_H:\,\mu \mapsto \wt\mu$, such that $\wt{\mu^{-1}}=(\wt\mu)^{-1}$ and for any $D\in W_r$
  \begin{equation}\label{induced-automorphism}
    \beta\big(\Ad_{\mu}(D)\big)=\Ad_{\wt\mu}\big(\beta(D)\big).
  \end{equation}
If \eqref{induced-automorphism} is true, then for any $\rho\in k[H_n]^{G_H},$
\begin{eqnarray*}
  \Big(\Ad_{\mu}\big(\wt\beta(\rho)\big)\Big)(D)&=&\rho\Big(\beta\big(\Ad_{\mu^{-1}}(D)\big)\Big)
=\rho\Big(\Ad_{\wt{\mu}^{-1}}\big(\beta(D)\big)\Big)\\
 &=&\big(\Ad_{\wt\mu}(\rho)\big)\big(\beta(D)\big)=\rho\big(\beta(D)\big)
=\wt\beta(\rho)(D).
\end{eqnarray*}
This means that $\wt\beta(\rho)\in k[G_W]^{G_W}.$

Let's go to prove \eqref{induced-automorphism}.
For any $\Ad_{\mu} \in G_W$ with $\mu\in \Aut B_r$,
let
$$\det\left(\left(\frac{\d \mu(x_j)}{\d x_i}\right)_{ij}\right)\equiv \alpha ~\mod{\mathfrak{m}}.$$
Then we define $\wt\mu \in \Aut B_n$ as follows.
\begin{eqnarray*}
\wt\mu(x_i)&=&\alpha \sum\limits_{k=1}^r x_k\left(\frac{\d\mu(x_k)}{\d x_i}\right)(\wt\mu(x_{r+1}),\cdots,\wt\mu(x_{2r}));\\
\wt\mu(x_{i+r})&=&\mu^{-1}(x_i)(x_{r+1},\cdots,x_{2r});\text{~~for~}1\leq i\leq r.
\end{eqnarray*}
It can be easily checked that $\wt{\mu^{-1}}(\wt\mu(x_i))=x_i$ for any $1\leq i \leq 2r.$
Then $$\wt{\mu^{-1}}=(\wt\mu)^{-1}.$$
\begin{lemma}\label{H-automorphism}
The automorphism $\Ad_{\wt\mu}$ belongs to $G_H$.
\end{lemma}

\begin{proof}
 We have to show that $\wt\mu(\omega_H)\in k^*\omega_H$. As
 \begin{eqnarray*}
  \wt\mu(\omega_H)&=& \sum\limits_{i=1}^r d\wt\mu(x_i)\wedge d\wt\mu(x_{i+r})\\
  &=& \alpha \bigg(\sum\limits_{k,i=1}^r\left(\frac{\d\mu(x_k)}{\d x_i}\right)\big(\wt\mu(x_{r+1}),\cdots,\wt\mu(x_{2r})\big) dx_k+\sum\limits_{k,i,t,j=1}^r x_k\left(\frac{\d^2\mu(x_k)}{\d x_t\d x_i}\right)\\
  &&\circ\big(\wt\mu(x_{r+1}),\cdots,\wt\mu(x_{2r})\big)\left(\frac{\d \mu^{-1}(x_t)}{\d x_j}\right)(x_{r+1},\cdots,x_{2r}) d x_{r+j}\bigg)\\
  && \wedge \sum\limits_{l=1}^r\left(\frac{\d \mu^{-1}(x_i)}{\d x_l}(x_{r+1},\cdots,x_{2r})d x_{r+l}\right)\\
  &=&\alpha \sum\limits_{k,l=1}^r \frac{\d \Big(\mu(x_k)\big(\mu^{-1}(x_1),\cdots,\mu^{-1}(x_n)\big)\Big)}{\d x_l}(x_{r+1},\cdots,x_{2r}) dx_k \wedge dx_{r+l}\\
  &&+\alpha \sum\limits_{1\leq j< l\leq r}(\Phi_{jl}-\Phi_{lj})dx_{j+r}\wedge dx_{l+r},
 \end{eqnarray*}
 where
$$\Phi_{jl}=\sum\limits_{k,i,t}^r x_k\left(\frac{\d^2\mu(x_k)}{\d x_t\d x_i}\right)\big(\wt\mu(x_{r+1}),\cdots,\wt\mu(x_{2r})\big)\left(\frac{\d \mu^{-1}(x_t)}
{\d x_j}\frac{\d \mu^{-1}(x_i)}{\d x_l}\right)(x_{r+1},\cdots,x_{2r}).$$
 Note that $\mu(x_k)(\mu^{-1}(x_1),\cdots,\mu^{-1}(x_n))=x_k$ and $\Phi_{jl}=\Phi_{lj}$, thus
 \begin{equation}\label{coefficent-of-automorphism}
   \wt\mu(\omega_H)=\alpha \omega_H.
 \end{equation}
 Therefore, our lemma is proved.
\end{proof}
Recall the relation between the groups $G_H$ and $G_B$ in \cite{St}.
\begin{theorem}\cite[Theorem\,7.3.6]{St}\label{automorphism-relation}
For $\mu\in \Aut B_n$ the following assertions are equivalent.
\begin{itemize}
\item[(a)] $\Ad_{\mu}\in G_H$.
\item[(b)] There is $\alpha\in k^*$ such that $[\mu(x_i),\mu(x_j)]=\alpha \sigma(i)\delta_{i'j}$, for all $f,g\in B_n$, where $\sigma(i)=1$ if $i\leq r$ and -1 otherwise; $i'=i+r$ if $i\leq r$ and equals $i-n$ otherwise.
\item[(c)] There is $\alpha\in k^*$ such that $[\mu(f),\mu(g)]=\alpha\mu[f,g]$, for all $f,g\in B_n$.
\end{itemize}
If (b) or (c) holds, then $\Ad_{\mu}(D_H(f))=D_H(\alpha^{-1}\mu(f))$ holds for all $f\in B_n$.
\end{theorem}

For our induced automorphism $\Ad_{\wt\mu}$, we have the following corollary.
\begin{corollary}\label{action-on-hn}
  For any $D_H(f)\in H_n$, we have $\Ad_{\wt\mu}(D_H(f))=\alpha^{-1}D_H(\wt\mu(f))$.
\end{corollary}

\begin{proof}
Note that
\begin{eqnarray*}
[\wt\mu(x_{r+i}),\wt\mu(x_i)]&=&D_H\big(\wt\mu(x_{r+i})\big)\left(\alpha \sum\limits_{k=1}^r x_k\left(\frac{\d\mu(x_k)}{\d x_i}\right)\big(\wt\mu(x_{r+1}),\cdots,\wt\mu(x_{2r})\big)\right)\\
  &=&-\alpha \sum\limits_{l=1}^r \frac{\d\wt\mu(x_{r+i})}{\d x_{r+l}}\left(\frac{\d\mu(x_l)}{\d x_i}\right)\big(\wt\mu(x_{r+1}),\cdots,\wt\mu(x_{2r})\big)\\
  &=&-\alpha \left(\frac{\d \mu^{-1}(x_{i})\big(\mu(x_1),\cdots,\mu(x_n)\big)}{\d x_{i}}\right)\big(\wt\mu(x_{r+1}),\cdots,\wt\mu(x_{2r})\big).
\end{eqnarray*}
As $\mu^{-1}(x_{i})\big(\mu(x_1),\cdots,\mu(x_n)\big)=x_i$,
then $[\wt\mu(x_{r+i}),\wt\mu(x_i)]=-\alpha.$
And hence by Theorem \ref{automorphism-relation}, we get
$$\Ad_{\wt\mu}\big(D_H(f)\big)=\alpha^{-1}D_H\big(\wt\mu(f)\big).$$
\end{proof}

\begin{proof}[Proof of \eqref{induced-automorphism}]
   For any $\rho\in k[H_n]^{G_H},$
  and any $\Ad_{\mu}\in G_W$, then we have an automorphism $\Ad_{\wt\mu}\in G_H$ defined as above. Hence for any $D=\sum\limits_{i=1}^r f_i\d_i \in W_r$, we have
  \begin{eqnarray*}
    \beta\big(\Ad_{\mu}(f_i\d_i)\big)&=&\beta\left(\sum\limits_{k=1}^r\Big(f_i\frac{\d \mu(x_k)}{\d x_i}\Big)\big(\mu^{-1}(x_1),\cdots,\mu^{-1}(x_r)\big)\d_k\right)\\
    &=&D_H\left(\sum\limits_{k=1}^r x_k\Big(f_i\frac{\d \mu(x_k)}{\d x_i}\Big)\big(\wt\mu(x_{r+1}),\cdots,\wt\mu(x_{2r})\big)\right)\\
    &=& \alpha^{-1} D_H\Big(f_i\big(\wt\mu(x_{r+1}),\cdots,\wt\mu(x_{2r})\big)\wt\mu(x_i)\Big)\\
    \Ad_{\wt\mu}\big(\beta(f_ix_i)\big)&=&\Ad_{\wt\mu}\Big(D_H\big(x_if_i(x_{r+1},\cdots,x_{2r})\big)\Big)\\
    &=& \alpha^{-1} D_H\Big(\wt\mu\big(x_if_i(x_{r+1},\cdots,x_{2r})\big)\Big)
\qquad\text{by~Corollary~\ref{action-on-hn},}\\
    &=& \alpha^{-1}D_H\Big(\wt\mu(x_i)f_i(\wt\mu(x_{r+1}),\cdots,\wt\mu(x_{2r})\Big)
  \end{eqnarray*}
  Thus $\beta(\Ad_{\mu}(f_i\d_i))=\Ad_{\wt\mu}(\beta(f_ix_i))$. By linearity, \eqref{induced-automorphism}
is true.
\end{proof}

To show that $\wt\beta$ is injective, let us first recall some subsets in $H_n$.
Recall that $V$ defined in \eqref{V} is an open subset in $H_n$.
Let
 $$S_0=\left\{\lambda+u+(-1)^{r-1}hx_{r+1}^{p-1}\cdots x_{2r}^{p-1}~|~\lambda\in k,~h\in \mathfrak{m}\cap k[x_1,\cdots,x_r]\right\},$$
  here $u=\sum\limits_{i=1}^r (-1)^{i-1}x_i\prod\limits_{j=1}^{i-1}x_{r+j}^{p-1}$.
  The set $G_B(f)\cap S_0$
  consists of only one element for every $f\in U$ due to \cite[Theorem\,5.2]{S}.
  Then for any $D_H(f)\in V$, there is an $s\in S_0$ such that
  $$D_H(s)\in G_H(D_H(f)).$$
\begin{proof}[Proof of Proposition \ref{propiso}]
We have proved that $\wt\beta(k[H_n]^{G_H})\subseteq k[W_r]^{G_W}$.
Let's start to prove $\wt\beta$ is injective.
If $\wt\beta(\rho)=0$ for $\rho\in k[H_n]^{G_H}$, then $\rho\big(\beta(W_r)\big)=0.$
For any $D_H(f)\in V$, there is an $s\in S_0$ such that $D_H(s)\in G_H(D_H(f))$.
Assume that $s=u+(\sum\limits_{i=1}^r \lambda_i x_i+h')x_{r+1}^{p-1}\cdots x_{2r}^{p-1}$
with $\deg h'\geq 2$ and $\lambda_i\in k$.
Let $\Ad_{\mu}\in G_H$ be that
$\mu(x_i)=ax_i$ and $\mu(x_{r+i})=x_{r+i}$ with $a\in k^*$ for $1\leq i \leq r$.
Then
\begin{eqnarray*}
  \Ad_{\mu}\big(D_H(s)\big)&=&D_H\big(a^{-1}\mu(s)\big)=D_H\Big(u+a^{-1}\mu\big(\sum\limits_{i=1}^r \lambda_i x_i+h'\big)x_{r+1}^{p-1}\cdots x_{2r}^{p-1}\Big)\\
  &=&D_H\Big(u+\big(\sum\limits_{i=1}^r \lambda_i x_i+ah^{''}\big)x_{r+1}^{p-1}\cdots x_{2r}^{p-1}\Big)
\end{eqnarray*}
for some $h^{''}\in \mathfrak{m}\cap k[x_1,\cdots,x_r].$

Now by taking limit as $a$ goes to $0$, it follows that
\begin{eqnarray*}
  D_H\Big(u+\sum\limits_{i=1}^r \lambda_i x_i\prod \limits_{i=1}^rx_{r+i}^{p-1}\Big)&\in & \overline{G_H(D_H(s))} \subseteq \overline{G_H(D_H(f))}
\end{eqnarray*}
 Since $\rho$ is $G_H$-invariant, it must be constant on
closures of orbits. Then
$$\rho(D_H(f))=\rho\left(D_H\Big(u+\sum\limits_{i=1}^r \lambda_i x_i\prod \limits_{i=1}^rx_{r+i}^{p-1}\Big)\right).$$
After this, note that
$$D_H\Big(u+\sum\limits_{i=1}^r \lambda_i x_i\prod \limits_{i=1}^rx_{r+i}^{p-1}\Big)=\beta\Big(\sum\limits_{i=1}^r \big((-1)^{i-1}\prod\limits_{j=1}^{i-1}x_j^{p-1}+\lambda_i\prod\limits_{i=1}^rx_i^{p-1}
\big)\d_i\Big),$$
and $\rho(\beta(W_r))=0$, hence $\rho(D_H(f))=0.$ Therefore, the map $\wt\beta$ is injective.
\end{proof}

\textbf{Acknowledgement} ~~The author would like to thank the Max Planck Institute 
for Mathematics in Bonn for providing the opportunity and good environment to study here.  The author also would like to thank Professor A.A.Premet for discussing his Conjecture with her.

\vspace{0.5cm}
\end{document}